\documentclass[3p]{elsarticle}

\pdfoutput=1

\makeatletter
\def\ps@pprintTitle{
    \let\@oddhead\@empty
    \let\@evenhead\@empty
    \def\@oddfoot{}
    \let\@evenfoot\@oddfoot}
\makeatother

\journal{}

\usepackage{graphicx}
\usepackage{amsmath,amsfonts,amsthm,amscd,mathrsfs,amssymb,upref,multirow,color}

\usepackage[utf8]{inputenc}
\usepackage{comment}
\usepackage[dvipsnames]{xcolor}
\usepackage{textcomp}

\newcommand{\R}{\mathbb{R}}
\newcommand{\N}{\mathbb{N}}

\newcommand{\e}{\mathrm{e}}

\newcommand{\vertiii}[1]{{\left\vert\kern-0.25ex\left\vert\kern-0.25ex\left\vert #1 \right\vert\kern-0.25ex\right\vert\kern-0.25ex\right\vert}}

\newcommand{\fucik}{Fu\v{c}\'{\i}k}

\usepackage[labelformat=simple]{subcaption}

\newtheorem{theorem}{Theorem}
\newtheorem{lemma}[theorem]{Lemma}

\newtheorem{corollary}[theorem]{Corollary}
\theoremstyle{definition} 
\newtheorem{remark}[theorem]{Remark}

\newtheorem*{acknowledgement}{Acknowledgement}

\begin{document}
	
\begin{frontmatter}
	
	\title{Notes on number of one-troughed travelling waves in asymmetrically supported bending beam}
	
	\author{Hana Formánková Lev\'a}
	\ead{levah@kma.zcu.cz}
	\author{Gabriela Holubov\'{a}}
	\ead{gabriela@kma.zcu.cz}
	
	\address{Department of Mathematics and NTIS, Faculty of Applied Sciences, University of~West Bohemia in Pilsen, Univerzitn\'{\i}~8, 301~00~Plze\v{n}, Czech Republic}
	
	\begin{abstract}
		We study the boundary value problem for nonlinear fourth-order partial differential equation with jumping nonlinearity which can serve, e.g., as a model of an asymmetrically supported bending beam. We focus on a special type of solutions, the so-called one-troughed travelling waves. The main goal of this paper is to show the existence of at least two different one-troughed travelling waves for particular wave speeds and input parameters of the studied problem. We present the upper bounds for the maximal number of one-troughed solutions together with a visualisation of obtained results and corresponding solutions. Finally, we list several open questions regarding this topic.
	\end{abstract}
		

\end{frontmatter}


\section{Introduction}
\label{sec:introduction}

Let us consider the following fourth order partial differential equation
\begin{equation}
\nonumber
\label{eq:general_eq}
u_{\tau\tau}+u_{xxxx}+ f_1(u) = f_2(x,\tau), \quad x \in \R, \, \tau>0,
\end{equation}
with special choice $f_1(u) = a u^+ - b u^-$ and $f_2(x,\tau) = 1$, where $u^+$ and $u^-$ stand for the positive and negative parts of $u$, i.e., $u^{\pm}(x,\tau) = \max\{\pm u(x,\tau),0\}$, $u = u^+-u^-$, and $a > b \geq 0$.  That is, we consider the above equation with a simple jumping nonlinearity and a constant (normalized) right-hand side. It can serve, e.g., as a model of an asymmetrically supported vibrating beam loaded only by its own weight. The unknown function $u=u(x,\tau)$ then describes the downward vertical deflection of the beam at a point $x$ and time $\tau$.

To be more specific, we will study the following boundary value problem
\begin{equation}
\label{eq:original_problem}
\left\{ \begin{array}{l}
u_{\tau\tau}+u_{xxxx}+ a u^+ - b u^- = 1, \quad x \in \R, \, \tau>0,\\
u(x,\tau) \rightarrow \dfrac{1}{a}, ~u_{x}(x,\tau) \rightarrow 0~~\textup{for}~|x| \rightarrow +\infty,
\end{array} \right.
\end{equation}
and we focus on its travelling wave solutions. One of the main motivations to study problems of type \eqref{eq:original_problem} was the collapse of Tacoma Narrows Bridge in 1940 only several months after its commissioning. It is usually mentioned in this context, however the travelling waves were also observed on other suspension bridges, e.g., Golden Gate Bridge in San Francisco as reported in \cite{amann_karman_woodruff}. Occurance of these phenomena in long structures as beams can be dangerous, therefore it is suitable to predict and prevent them. 

Since Tacoma Narrows Bridge collapse many papers have been published dealing with this topic. As far as we know, the main series of them was started by McKenna and Walter \cite{mckenna_walter} in 1990. The authors dealt with a normalized model of a suspension bridge where the searched solution $u$ represents the deflection of the roadbed of the bridge and the term $au^+-bu^-$ corresponds to the restoring force of the bridge cables. They behave like springs with different stiffness with respect to opposite directions. The coefficient $a$ gives the downward stiffness, i.e. the stiffness in the positive direction, $b$ describes stiffness for moving upwards. The authors studied the problem \eqref{eq:original_problem} with a special choice  $a=1$, $b=0$, i.e., the cables do not cause the upward movement of the roadbed. They proved the existence of the travelling wave solution by finding its analytical form. At the same time they introduced the first limitation on the wave speed to the interval $\left(0, \sqrt{2}\right)$ for which there exists the required solution. Notice that this corresponds to $\left(0, \sqrt[4]{4a}\right)$ for a general $a>0$.

Their work was followed by paper by Chen and McKenna \cite{chen_mckenna} who proved the existence of travelling wave solution by variational methods, in particular Mountain Pass Theorem and concentration compactness principle. They also presented results for the problem \eqref{eq:original_problem} with an additional nonlinear term $g(u)$ satisfying certain suitable conditions, i.e., with $f_1(u) = au^++g(u)$. The Mountain Pass algorithm was also used to find some of the solutions numerically. Moreover, Chen and McKenna studied stability of travelling waves together with other properties and observed surprising behaviour similar to solitons (after collision, two waves appear almost intact).

The topic was further extended by Chen \cite{chen} who considered other types of the nonlinearity $f_1$ instead of the piecewise linear one. In particular, the variational proof for problems with the so called fast increasing nonlinearity was presented. However, the used assumptions on the nonlinear term were quite restrictive.

Later, Lazer and McKenna in \cite{lazer_mckenna} contributed to the topic by showing the unboundedness of the wave amplitude for the wave speed converging to zero. It confirmed the observations from the numerical experiments in the original paper \cite{mckenna_walter}. Moreover, Karageorgis and Stalker \cite{karageorgis_stalker} presented a lower bound for the amplitude of the travelling waves. Results concerning stability and interaction properties can be found, e.g., in \cite{champneys_mckenna_zegeling} or \cite{levandosky}. As further contributions to the topic we can mention the following papers \cite{diaferio_sepe}, \cite{camacho_bruzon_ramirez_gandarias}, \cite{li_sun_wu}, \cite{lazer_mckenna_2} and references therein. 

For engineers and for the application in practice it is more convenient to use the ``smoothed'' nonlinearity  $e^{au-1}$ instead of $au^+$. Both functions have the same value at $u=1/a$ and same limits for $u \to \pm \infty$, however, the advantage of the exponential is the existence of its derivative everywhere. Such nonlinear terms are more suitable for numerical experiments despite their superlinear growth. Namely, we mention the work by Chen and McKenna \cite{chen_mckenna_2}. They studied the existence of the solutions of the problem with exponential nonlinearity and their stability and interaction properties. The authors further summarize several open questions concerning the travelling waves in suspension bridges, some of which remain open until today.

Regarding the number of solutions, Smets and van den Berg \cite{smets_vandenberg} proved the existence of at least one solution of the problem with the exponential nonlinearity with $a =1$ for almost all the wave speeds from the interval $\left(0, \sqrt{2}\right)$. Next, in \cite{breuer} computer assisted proof gave 36 different solutions with fixed wave speed. 

\medskip

In \cite{holubova_leva}, the existence of homoclinic travelling wave solutions of \eqref{eq:original_problem} with $b > 0$ and an additional nonlinear term $g(u)$ was studied. The authors consider more general assumptions on $g$ then previously in \cite{chen_mckenna}. Moreover, allowing the presence of the term $bu^-$ results in the limitation of the possible wave speeds. In contrast to Chen and McKenna's necessary condition for the existence of travelling wave solution $|c| \in (0,\sqrt[4]{4a})$ from \cite{chen_mckenna}, i.e., for the problem \eqref{eq:original_problem} with $b=0$, the necessary condition in \cite{holubova_leva} reads as $|c| \in (\sqrt[4]{4b},\sqrt[4]{4a})$. Both are in accordance with the results in \cite{karageorgis_stalker}.

Using variational methods in \cite{holubova_leva}, the existence of travelling waves with wave speed from the interval $(\sqrt[4]{100b/9},\sqrt[4]{4a})$ was proved in \cite{holubova_leva}. However, the authors admit that the interval could be enlarged, which is the subject of the paper \cite{holubova_leva_necesal}. There was presented the maximal wave speed range $(\sqrt[4]{4b/\beta^{\ast}},\sqrt[4]{4a})$ with $\beta^{\ast} = \beta^{\ast}(a/b) \in (b/a, 1)$ for which the existence of the travelling wave solution can be proved using Mountain Pass Theorem. Since the lower bound is given by the \fucik ~spectra of the related Dirichlet problems which are difficult to compute, the authors introduced its several approximations and its relation to the \fucik ~spectra of a periodic problems. In Figure \ref{fig:beta_star} there are two illustrations of the necessary and sufficient conditions for the existence of travelling wave solutions from \cite{holubova_leva} and \cite{holubova_leva_necesal} regarding the wave speed  $\sqrt[4]{a}|c|$ (see the next section for the explanation of this setting).

\begin{figure}[thb]
\centerline{
  \setlength{\unitlength}{1mm}
  \begin{picture}(80, 47)(0,0)
    \put(0,0){\includegraphics[height=4.7cm]{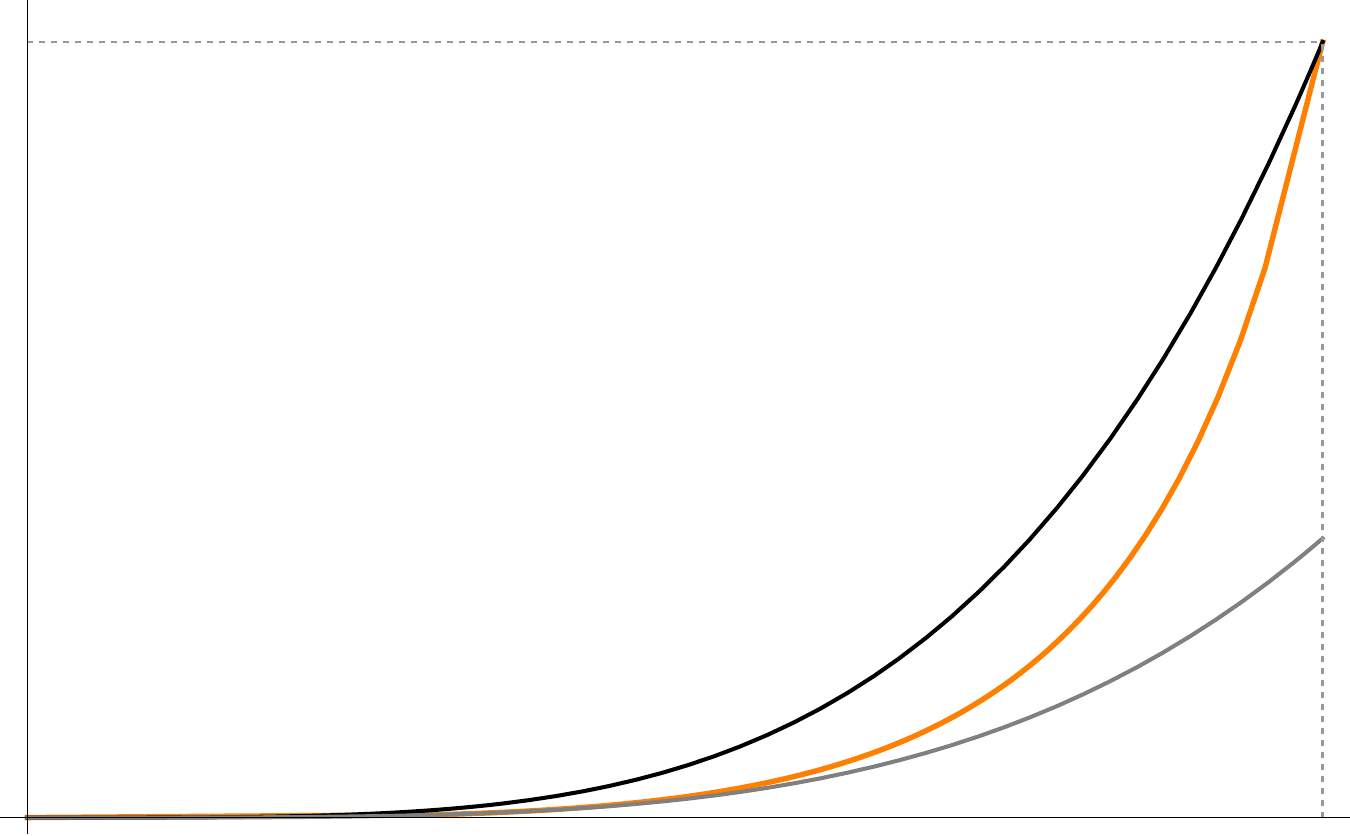}}       
	\put(78,0){\makebox(0,0)[lb]{\footnotesize$c$}}       
    \put(-2,47){\makebox(0,0)[lb]{\footnotesize$\frac{b}{a}$}}  
    \put(65,35){\makebox(0,0)[lb]{\footnotesize$\frac{c^4}{4}$}}       
    \put(70,25){\makebox(0,0)[lb]{\footnotesize$\beta^{\ast}\frac{c^4}{4}$}} 
    \put(75,15){\makebox(0,0)[lb]{\footnotesize$\frac{9c^4}{100}$}}  
    \put(-1,44){\makebox(0,0)[lb]{\tiny$1$}}      
    \put(73,-2){\makebox(0,0)[lb]{\tiny$\sqrt{2}$}}
    \put(0,-1){\makebox(0,0)[lb]{\tiny$0$}}   
  \end{picture}
  \hspace{0.25cm}
  \begin{picture}(80, 47)(0,0)
      \put(0,0){\includegraphics[height=4.7cm]{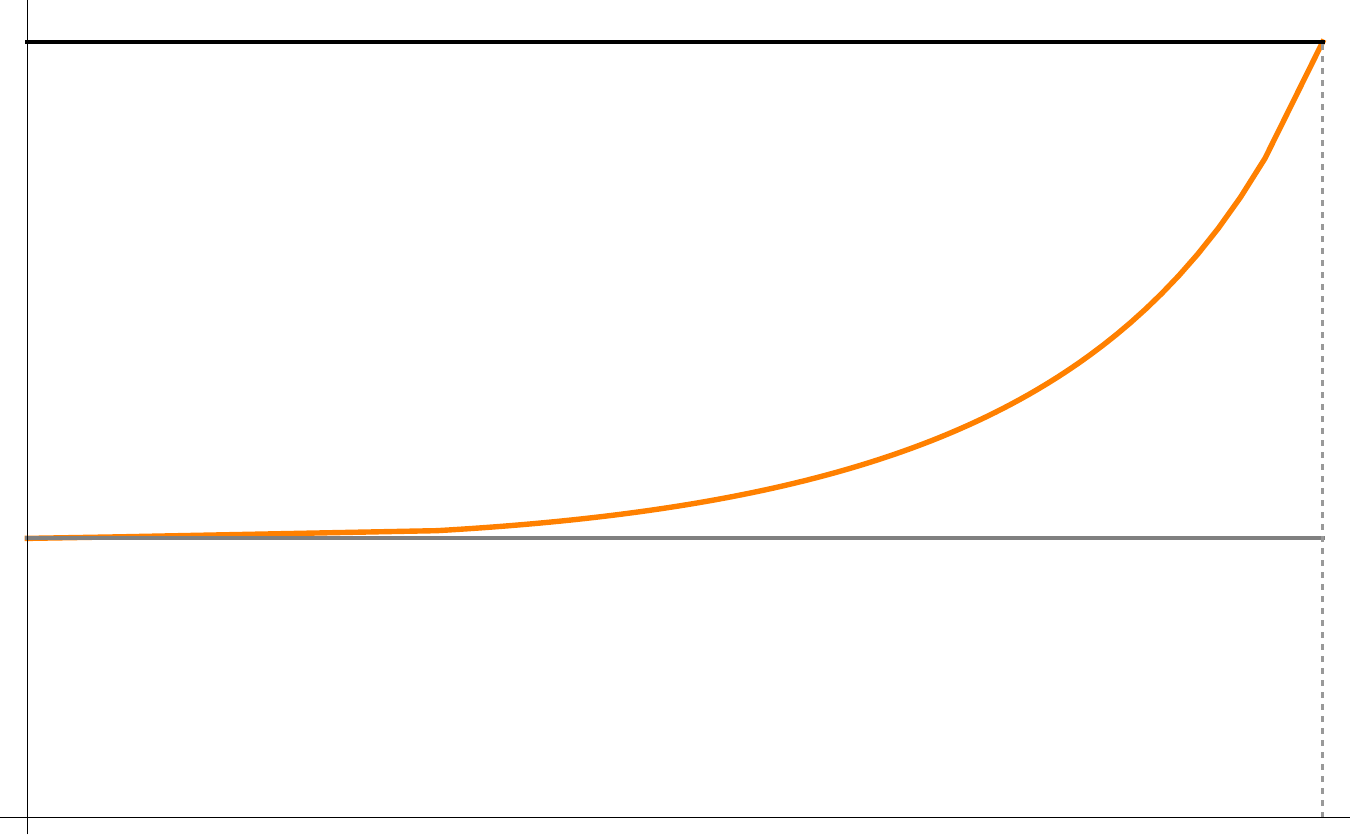}} 
      \put(78,0){\makebox(0,0)[lb]{\footnotesize$c$}}       
      \put(-2,47){\makebox(0,0)[lb]{\footnotesize$p$}}
      \put(69,31){\makebox(0,0)[lb]{\footnotesize$\beta^{\ast}$}}
      \put(-1,44){\makebox(0,0)[lb]{\tiny$1$}}      
      \put(-2,15){\makebox(0,0)[lb]{\tiny$\frac{9}{25}$}}    
      \put(73,-2){\makebox(0,0)[lb]{\tiny$\sqrt{2}$}} 
      \put(0,-1){\makebox(0,0)[lb]{\tiny$0$}}  
  \end{picture}    
  }
\caption{On the left, the illustration of  the border line of the area corresponding to the necessary condition $b/a < c^4/4 < 1$ (in black) and sufficient conditions $b/a < 9c^4/100$ from \cite{holubova_leva} (in gray) and $b/a < \beta^{\ast}(a/b)c^4/4$ from \cite{holubova_leva_necesal} (in orange) for the existence of travelling wave solution. On the right, the left picture is rescaled onto a rectangle by $b/a = p c^4/4$ with $p \in (0,1)$.}
\label{fig:beta_star}
\end{figure}
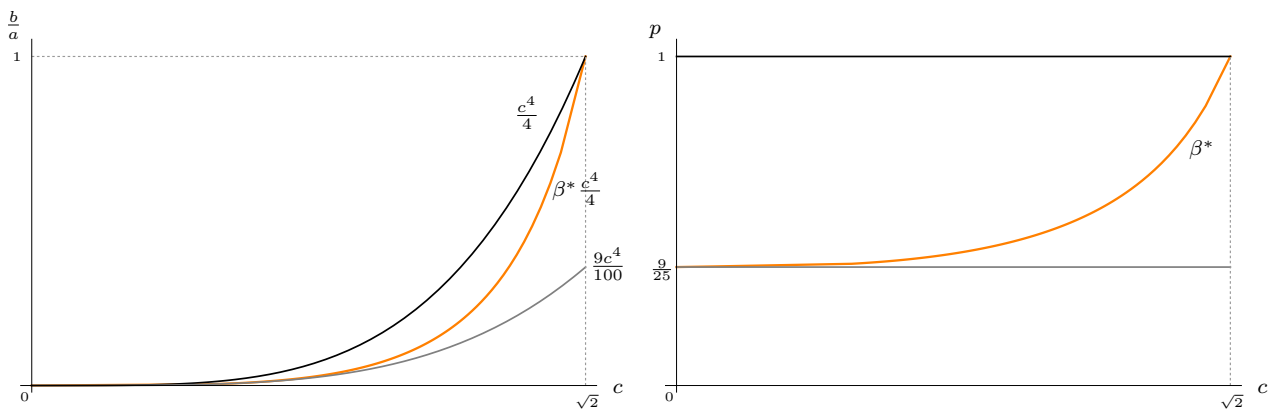

In this paper we would like to (partially) address the problem of number of the solutions of \eqref{eq:original_problem} with fixed wave speed.   Earlier, Champneys and McKenna \cite{champneys_mckenna} proved that there are at most 5 different travelling wave solutions in special form with fixed wave speeds from some nested subintervals of the range of wave speed in the case $b=0$. Upper and lower bounds of these intervals were found numerically by Maple.


\section{One-troughed travelling wave solutions}
\label{sec:symmetric_one_troughed}

In this section we will focus on a special type of solutions of \eqref{eq:original_problem} which can be described analytically thanks to the piecewise nonlinearity. The subject of our research is the number of such solutions that are different from each other for any fixed travelling wave speed.

First of all, we will substitute
\begin{equation*}
a u(x,\tau) = y(\sqrt[4]{a} x-c \sqrt{a} \tau) = y(t)
\end{equation*}
into the PDE in \eqref{eq:original_problem} and thereby transform it to an ODE with equilibrium $y \equiv 1$. Let us note, that $\sqrt[4]{a}|c|$ corresponds to the wave speed. With $c>0$ the wave travels to the right, with $c<0$ it travels to the left. From now on, for the simplicity, we will consider only $c>0$. Our results hold also for $-c$.

Further, we translate the stationary state to zero by putting $z(t) := y(t) - 1$. We also introduce the parameter $\xi = b/a$, $\xi \in \left[0,{1}\right)$. Thus, we come to the problem
\begin{equation}
\label{eq:transformed_ode}
\left\{ 
\begin{array}{ll} 
z^{(4)}+c^2 z''+ (z+1)^{+}-\xi (z+1)^{-} - 1 = 0,\\
z, \, z' \rightarrow 0~~\textup{for}~|t| \rightarrow +\infty.  
\end{array} 
\right.
\end{equation}
The ODE in \eqref{eq:transformed_ode} can also be rewritten as two linear ordinary differential equations
\begin{eqnarray}
\label{eq:first_ode}
z^{(4)}+c^2 z''+z = 0 &\qquad &\textup{for}~z \geq -1,\\
\label{eq:second_ode}
z^{(4)}+c^2 z''+\xi z = 1-\xi &\qquad &\textup{for}~z \leq -1.
\end{eqnarray}

The homoclinic solutions of \eqref{eq:transformed_ode} will be constructed as the solution of \eqref{eq:first_ode} decaying to zero for $|t| \rightarrow +\infty$ connected with the solution of \eqref{eq:second_ode} at points $t$ that satisfy $z(t) = -1$. We will be interested only in the case $c^2 \in \left(0,{2}\right)$ and $\xi = b/a \in \left[0, c^4/4\right)$ due to the existence result and the necessary condition in \cite{holubova_leva} or \cite{karageorgis_stalker} (cf. Figure \ref{fig:beta_star}).

There are several options how to distinguish between different travelling wave solutions. By various solutions we mean two waves that are not only a translation of each other but there is some basic difference between them. Champneys and McKenna in \cite{champneys_mckenna} introduce two ways of characterizing the solutions. In particular, first way is to count the number of times that the solution drops below the value $-1$. Such solutions are referred as \textit{multi-troughed} waves. An \textit{one-troughed} wave is a solution that drops under $-1$ only on one interval. The latter way is to count the number of local maxima of the solution in the region where $z \leq -1$. It is usually used to distinguish between several one-troughed waves. These solutions are referred as \textit{multiwiggle} waves.

In this section we will focus on the (even) one-troughed travelling waves using the terminology introduced in \cite{champneys_mckenna}, i.e., the solution for which there are just two points $t$ with $z(t)=-1$. First, we will look for their analytical description. The main goal is to show that there exist more than one one-troughed solution for some fixed admissible wave speeds.

The following paragraph is highly inspired by \cite{champneys_mckenna}. For the simplicity, we will focus on even solutions. First, the general solution of \eqref{eq:first_ode} that satisfies $z(t)\rightarrow 0$ for $t \rightarrow -\infty$ can be written as
\begin{equation}
\label{eq:general_sol_of_1}
z_1(t) = A \e^{\lambda t}\cos (\omega t + B)
\end{equation}
with $A$ and $B$ being free parameters and
\begin{equation}
\label{eq:lambda_omega}
\lambda = \dfrac{\sqrt{2-c^2}}{2} \qquad \textup{and} \qquad \omega = \dfrac{\sqrt{2+c^2}}{2}.
\end{equation}
The derivatives of \eqref{eq:general_sol_of_1} have the form
\begin{eqnarray}
\label{eq:first_der_of_1}
z_1'(t)&=& A \e^{\lambda t} \left(\lambda \cos (\omega t + B) - \omega \sin (\omega t + B)\right),\\
\label{eq:second_der_of_1}
z_1''(t) &=& A\e^{\lambda t} \left(\left(\lambda^2-\omega^2\right)\cos(\omega t + B)-2\lambda \omega \sin (\omega t + B)\right),\\
\label{eq:third_der_of_1}
z_1'''(t)&=&A\e^{\lambda t}\left(\left(\lambda^3-3\lambda\omega^2\right)\cos(\omega t + B)+ \left(\omega^3-3\lambda^2\omega\right)\sin (\omega t + B)\right).
\end{eqnarray}

For completeness of the text, we present here a one to one approach as in Section 2 of \cite{champneys_mckenna}. This is necessary to understand the follow-up part of our text. We will consider only a nontrivial solutions $z_1$ with $t \in \left(-\infty,0\right]$. We look for the smallest $t_1<0$ for which $z_1(t_1)=-1$. The solution $z_1$ will be connected with solution $z_2$ of \eqref{eq:second_ode} on the interval $[t_1,0)$ at the point $t_1$ and extended symmetrically and smoothly on $\left[ 0, + \infty\right)$. We parametrize $z_1'(t_1)=\theta$. Thus, we can write
\begin{eqnarray}
\label{eq:first_sol_parametrized}
\nonumber &z_1(t_1) = -1, \qquad z_1'(t_1) = \theta,\\
&z_1''(t_1) = 1 + \theta \sqrt{2-c^2}, \qquad z_1'''(t_1) = \sqrt{2-c^2} + \theta(1-c^2).
\end{eqnarray}
Now, we have to state the range for parameter $\theta$ in order to $z_1(t)\geq -1$ for all $t<t_1$. Since $\theta$ describes the slope of $z_1$ at the connection point $t_1$, it has to be negative for simple geometric reason. Thus, $\theta_{max}=0$. It remains to determine its lower bound $\theta_{min}$. This strict lower bound can be derived from computing the trajectory through the point parametrized by the upper bound $\theta_{max}
=0$ achieved at some $t_0$ until the next value $t_0+t^{\ast}$ such that $z_1(t_0+t^{\ast})=-1$ (see Figure \ref{fig:1}, left). Then $\theta_{min}:=z_1'(t_0+t^{\ast})$. Thus
\begin{eqnarray*}
-1 &=& z_1(t_0+t^{\ast}) = A\e^{\lambda(t_0+t^{\ast})}\cos\left(\omega(t_0+t^{\ast})+B\right),\\
\theta_{min} &=& z_1'(t_0+t^{\ast}) = A \e^{\lambda(t_0+t^{\ast})}\left(\lambda\cos\left(\omega(t_0+t^{\ast})+B\right)-\omega \sin\left(\omega(t_0+t^{\ast})+B\right)\right).
\end{eqnarray*}
Thanks to \eqref{eq:general_sol_of_1}, \eqref{eq:first_der_of_1}, \eqref{eq:first_sol_parametrized} and the fact that $z_1(t_0)=-1$ and $z_1'(t_0)=0$ we can write
\begin{eqnarray}
\label{eq:first_theta_min}
\e^{\lambda t^{\ast}}\left(\lambda \sin\left(\omega t^{\ast}\right)-\omega\cos\left(\omega t^{\ast}\right)\right)  &=& - \omega,\\
\label{eq:second_theta_min}
\e^{\lambda t^{\ast}}\left(\lambda^2+\omega^2\right)\sin\left(\omega t^{\ast}\right) &=& \omega \theta_{min}.
\end{eqnarray}
Using $\lambda^2+\omega^2=1$ (see \eqref{eq:lambda_omega}) and after a standard manipulation, we transform \eqref{eq:first_theta_min}--\eqref{eq:second_theta_min} into
\begin{eqnarray}
\label{eq:exp_theta_min}
\e^{2 \lambda t^{\ast}} & = & \theta_{min}^2 + 2\lambda \theta_{min} + 1,\\
\label{eq:cotg_theta_min}
\cot \left(\omega t^{\ast}\right) & = & \dfrac{\lambda \theta_{min}+1}{\omega \theta_{min}}.
\end{eqnarray}
Since $\omega t^{\ast} \in \left(\pi, 2\pi\right)$ (see Figure \ref{fig:1}, left) there exists exactly one solution of \eqref{eq:cotg_theta_min}
\begin{equation*}
t^{\ast} = \dfrac{1}{\omega}\left(\dfrac{3\pi}{2}-\arctan\left(\dfrac{\lambda \theta_{min}+1}{\omega \theta_{min}}\right)\right).
\end{equation*}
By substituting the value $t^{\ast}$ into \eqref{eq:exp_theta_min} we get unique $\theta_{min} = \theta_{min}(c)$. Its graph is depicted in Figure \ref{fig:1}, right. Up to this point, our calculations match with the ones stated in \cite{champneys_mckenna}.

\medskip

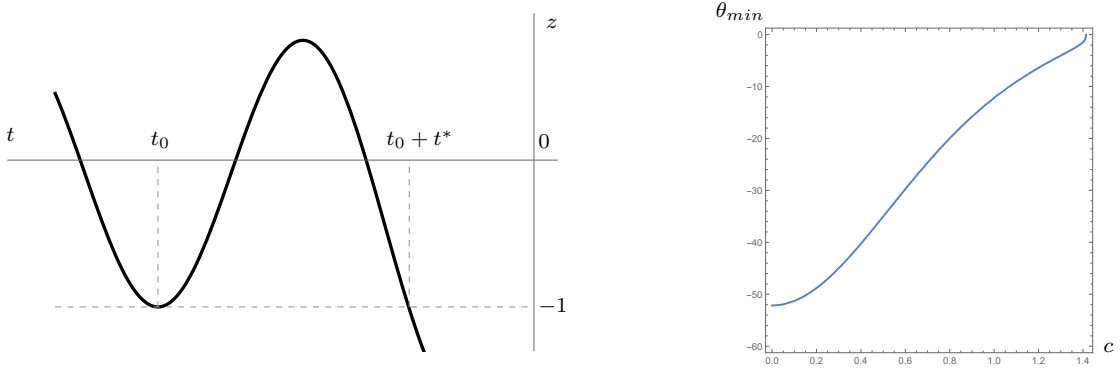
\begin{figure}[t]
\centerline{
  \setlength{\unitlength}{1mm}
  \begin{picture}(70, 45)(0,0)
    \put(0,2){\includegraphics[height=4.5cm]{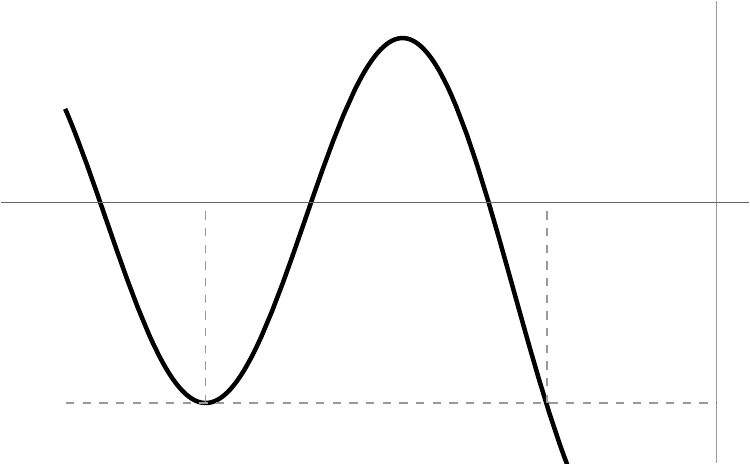}}
    \put(0,30){\makebox(0,0)[lb]{\footnotesize$t$}}         
    \put(71,45){\makebox(0,0)[lb]{\footnotesize$z$}} 
    \put(19,29){\makebox(0,0)[lb]{\footnotesize$t_0$}} 
    \put(50,29){\makebox(0,0)[lb]{\footnotesize$t_0+t^{\ast}$}} 
    \put(70,7){\makebox(0,0)[lb]{\footnotesize$-1$}} 
    \put(70,29){\makebox(0,0)[lb]{\footnotesize$0$}}            
  \end{picture} 
  \hspace{2.5cm}
  \begin{picture}(45, 45)(0,0)
    \put(0,0){\includegraphics[height=4.5cm]{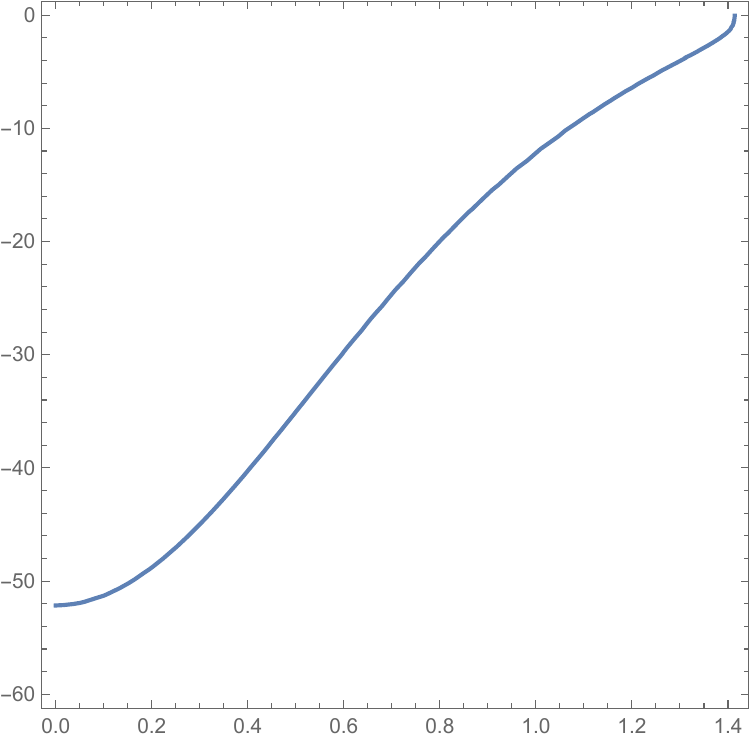}}
    \put(-4,46){\makebox(0,0)[lb]{\footnotesize$\theta_{min}$}}         
    \put(47,2){\makebox(0,0)[lb]{\footnotesize$c$}}            
  \end{picture}   
  }
\caption{On the left, the illustration of the process of finding $\theta_{min}$. On the right, the dependence of $\theta_{min}$ on $c$.}
\label{fig:1}
\end{figure}

Further, we write the even general solution of \eqref{eq:second_ode} in the following form
\begin{equation*}
z_2(t) = \gamma \cos\left(\kappa_1 t \right) + \delta \cos\left(\kappa_2 t \right) + \dfrac{1-\xi}{\xi}
\end{equation*}
where $\gamma, \,\delta, \in \R$,
\begin{equation*}
\kappa_1^2 = \dfrac{c^2-\sqrt{c^4-4\xi}}{2} \qquad \textup{and} \qquad \kappa_2^2 = \dfrac{c^2+\sqrt{c^4-4\xi}}{2}.
\end{equation*}
Thus $0<\kappa_1<\kappa_2$. Additionally, there exists some $t_2<t_1<0$ for which $z_2(t_2) = (1-\xi)/\xi$, i.e., $\gamma \cos (\kappa_1 t_2)=-\delta \cos (\kappa_2 t_2)$. Definitely, $t_2 > -\pi/\kappa_1$ and thus $t_1 \in (-\pi/\kappa_1,0)$. Next, we will match $z_2$ and its first three derivatives with the solution $z_1$ given by \eqref{eq:general_sol_of_1}, in particular we use \eqref{eq:first_sol_parametrized}. Then, we can write
\begin{eqnarray}
\label{eq:1}
\gamma \cos \left(\kappa_1 t_1\right) + \delta \cos \left(\kappa_2 t_1\right) + \frac{1-\xi}{\xi} &=& -1,\\
\label{eq:2}
-\gamma \kappa_1 \sin \left(\kappa_1 t_1\right)-\delta \kappa_2 \sin \left(\kappa_2 t_1\right) &=& \theta,\\
\label{eq:3}
-\gamma \kappa_1^2 \cos \left(\kappa_1 t_1\right) - \delta \kappa_2^2 \cos \left(\kappa_2 t_1\right) &=& 1 + \theta \sqrt{2-c^2},\\
\label{eq:4}
\gamma \kappa_1^3 \sin \left(\kappa_1 t_1\right) + \delta \kappa_2^3 \sin \left(\kappa_2 t_1\right) &=& \sqrt{2-c^2}+\theta \left(1-c^2\right).
\end{eqnarray}
By eliminating $\gamma$, $\delta$ and $t_1<0$ from \eqref{eq:1}--\eqref{eq:4}, we get equations
\begin{equation}
\label{eq:L_eq}
L(c,\xi,\theta) = k
\end{equation}
where
\begin{eqnarray}
\nonumber
L(c,\xi,\theta):=\dfrac{1}{2}\left(1+\frac{\kappa_2}{\kappa_1}\right)+\dfrac{1}{\pi}\arctan\left(\dfrac{\kappa_2(-\kappa_1^2+\xi+\xi\theta\sqrt{2-c^2})}{\xi (\theta \kappa_1^2 + \sqrt{2-c^2}+ \theta (1-c^2))}\right)\\
\nonumber
-\dfrac{\kappa_2}{\kappa_1 \pi}\arctan \left(\dfrac{\kappa_1(-\kappa_2^2+\xi+\xi\theta\sqrt{2-c^2})}{\xi (\theta \kappa_2^2 + \sqrt{2-c^2}+\theta (1-c^2))}\right),
\end{eqnarray}
$\theta \in (\theta_{min}(c),0)$ and $k \in \N$. Thus, we reduce the search for one-troughed travelling waves to~the~search for solutions of the system of equations \eqref{eq:L_eq}. For further use, we present here the parameters
\begin{eqnarray}
\nonumber
\gamma^2 &=& \dfrac{(\xi - 1)(c^2 + \sqrt{c^4-4\xi}+2\theta^2\xi)}{\xi(c^4-4\xi)(-c^2+\sqrt{c^4-4\xi})},\\
\nonumber
\delta^2 &=& \dfrac{(\xi - 1)(-c^2+\sqrt{c^4-4\xi}-2\theta^2\xi)}{\xi(c^4-4\xi)(c^2+\sqrt{c^4-4\xi})},\\
\nonumber
t_1 &=& \dfrac{1}{\kappa_1}\arctan\left(\dfrac{\kappa_1(-\kappa_2^2 + \xi + \xi\theta\sqrt{2-c^2})}{\xi(\theta\kappa_2^2 + \sqrt{2-c^2}+\theta(1-c^2)}\right)-\dfrac{\pi}{2\kappa_1} \qquad \in \left(-\dfrac{\pi}{\kappa_1},0\right).
\end{eqnarray}
It is easy to see that $|\gamma|>|\delta|$.


\section{Properties of $L$ and the first estimates on the number of one-troughed solutions}
\label{sec:number_one_troughed}

In the following part, we will examine the properties of $L$, which enable us to estimate the number of one-troughed travelling wave solutions of \eqref{eq:original_problem}. First of all, we rescale the parameter $\xi$ by $\xi = p c^4/4$ with $p \in (0,{1})$, and we use $L(c,p,\theta) := L(c,pc^4/4,\theta)$, for simplicity. There are two discontinuities of $L$
\begin{equation*}
\theta_{1,2} = \dfrac{-2\sqrt{2-c^2}}{2-c^2 \mp c^2 \sqrt{1-p}}.
\end{equation*}
From $z_2(0)<0$, $|\gamma|>|\delta|$ and $\kappa_1 t_1 \in (-\pi,0)$ it follows that $\theta < \theta_2<0$. Thus, the admissible parameters must satisfy
\begin{equation}
\label{eq:interval_theta}
\theta \in \left(\theta_{min}(c),\theta_2(c,p)\right).
\end{equation}
There does not exist any one-troughed travelling wave solution if the interval in \eqref{eq:interval_theta} is empty. See the orange area in Figure \ref{fig:4}, left. Moreover, for all admissible values of $c$, $p$ and $\theta$ the partial derivative $L_{\theta}'(c,p,\theta)<0$ which means that $L$ is decreasing  in $\theta$ on relevant subintervals of $(\theta_{min},\theta_2)$.

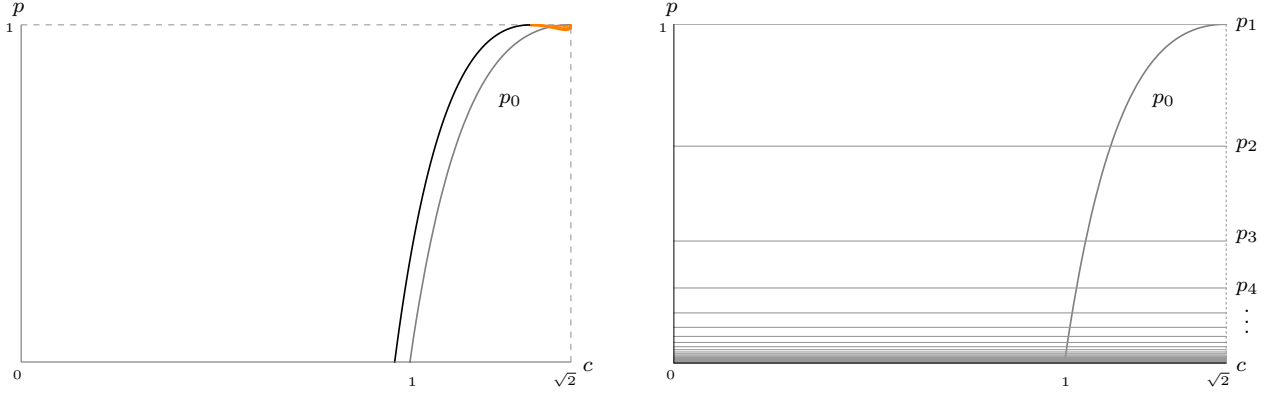
\begin{figure}[t]
\centerline{
  \setlength{\unitlength}{1mm}
  \begin{picture}(75, 45)(0,0)
    \put(1,1){\includegraphics[height=4.5cm]{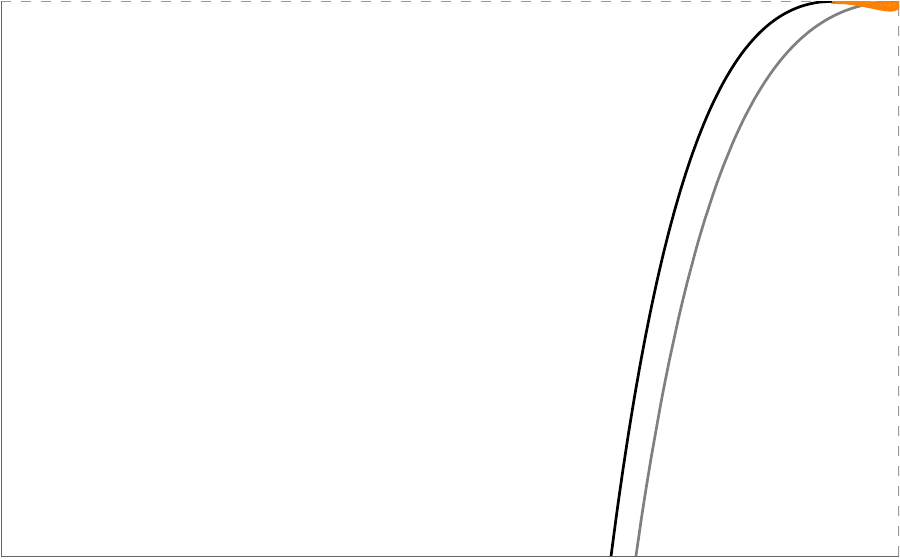}}  
    \put(75,0){\makebox(0,0)[lb]{\footnotesize$c$}}       
    \put(0,47){\makebox(0,0)[lb]{\footnotesize$p$}} 
    \put(64,35){\makebox(0,0)[lb]{\footnotesize$p_0$}}
    \put(0,-1){\makebox(0,0)[lb]{\tiny$0$}}
    \put(-1,45){\makebox(0,0)[lb]{\tiny$1$}}
    \put(71,-2){\makebox(0,0)[lb]{\tiny$\sqrt{2}$}}  
    \put(52,-2){\makebox(0,0)[lb]{\tiny$1$}}     
  \end{picture} 
  \hfill
  \begin{picture}(75, 45)(0,0)
    \put(1,1){\includegraphics[height=4.5cm]{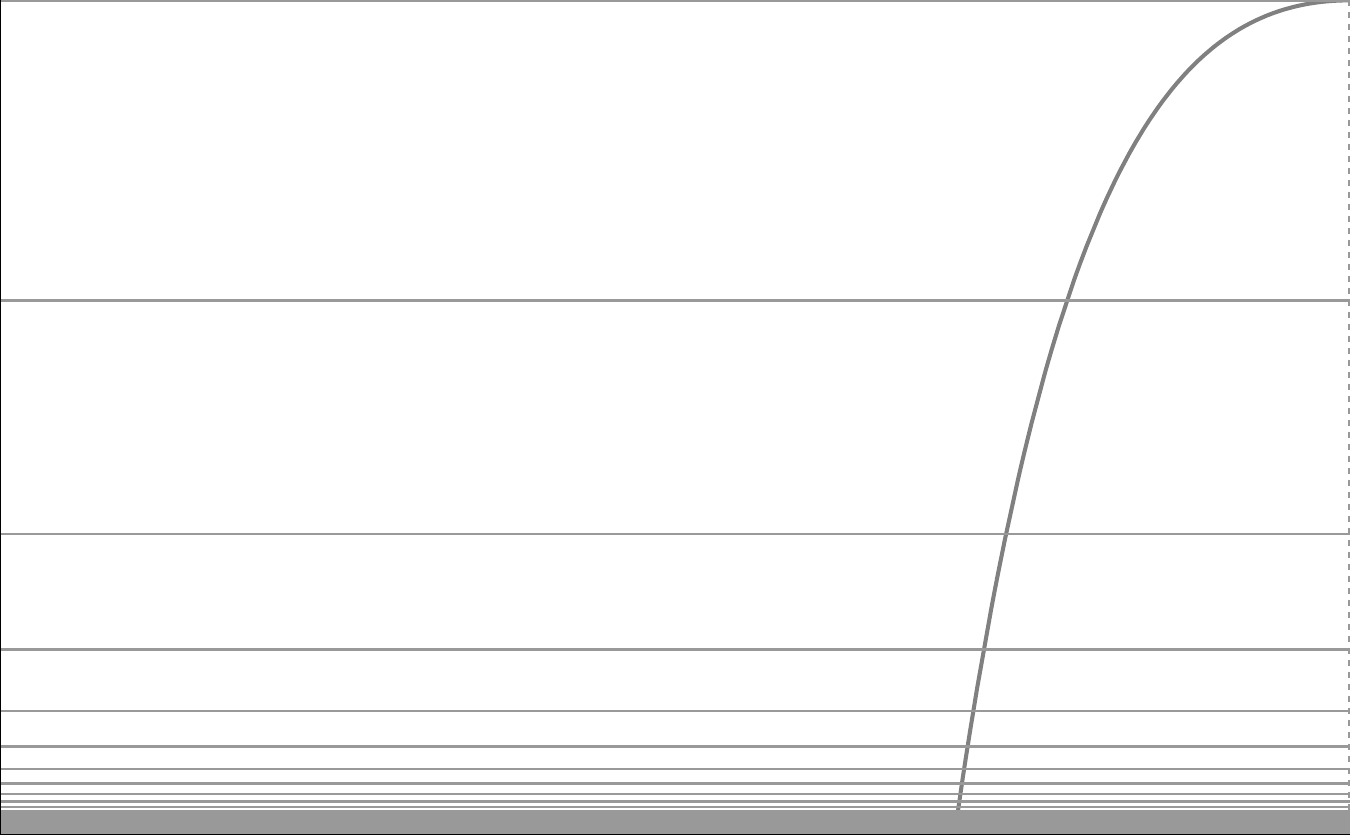}}       
	\put(75,0){\makebox(0,0)[lb]{\footnotesize$c$}}       
    \put(0,47){\makebox(0,0)[lb]{\footnotesize$p$}}  
    \put(64,35){\makebox(0,0)[lb]{\footnotesize$p_0$}} 
    \put(75,45){\makebox(0,0)[lb]{\footnotesize$p_1$}} 
    \put(75,29){\makebox(0,0)[lb]{\footnotesize$p_2$}} 
    \put(75,17){\makebox(0,0)[lb]{\footnotesize$p_3$}} 
    \put(75,10){\makebox(0,0)[lb]{\footnotesize$p_4$}} 
    \put(76,5){\makebox(0,0)[lb]{\footnotesize$\vdots$}}  
    \put(0,-1){\makebox(0,0)[lb]{\tiny$0$}}
    \put(-1,45){\makebox(0,0)[lb]{\tiny$1$}}
    \put(71,-2){\makebox(0,0)[lb]{\tiny$\sqrt{2}$}}  
    \put(52,-2){\makebox(0,0)[lb]{\tiny$1$}}   
  \end{picture}        
  }
\caption{On the left, the illustration of the empty interval for admissible $\theta$. The orange area corresponds to the situation where $\theta_{min}\geq \theta_2$, the black curve is determined by $\theta_{min}=\theta_1$ and the gray curve $p_0$ is given by \eqref{eq:p0}. On the right, the illustration of strips in the $cp$ plane given by \eqref{eq:pk} with different upper bounds for the number of solutions.}
\label{fig:4}
\end{figure}

\begin{lemma}
\label{lemma:p>p0}
Let $c \in \left(0,\sqrt{2}\right)$, $p \in \left(0,{1}\right)$ and $p>p_0$ where
\begin{equation}
\label{eq:p0}
p_0:= \dfrac{4}{c^4}(c^2-1).
\end{equation}
If ratio $\kappa_2/\kappa_1 <k$, $k \in \N$, then \eqref{eq:original_problem} has at most $2 \left(\lceil k/2\rceil -1\right)$ even one-trouhged travelling wave solutions.
\end{lemma}

\begin{proof}
Let us take any fixed $c\in \left(0,\sqrt{2}\right)$ and $p \in (0,1)$ such that $p>p_0$. Then $\theta_1<\theta_2<0$. We can estimate the range $R(L)$ by finding upper and lower bounds for values of $L$ for $\theta \in \left(-\infty,\theta_2\right)$. Function $L$ is continuous and decreasing in $\theta$ on intervals $\left(-\infty,\theta_1\right)$ and $\left(\theta_1, \theta_2\right)$. Simply,
\begin{equation*}
\lim_{\theta \rightarrow -\infty} L\left(c,p,\theta\right) = L(c,p,-\infty) = \dfrac{1}{2}\left(1+\dfrac{\kappa_2}{\kappa_1}\right)+\dfrac{1}{\pi}\arctan \left(-\theta_1 \kappa_2\right)-\dfrac{\kappa_2}{\pi \kappa_1}\arctan \left(-\theta_2\kappa_1\right).
\end{equation*}
Both arguments of arctangent functions are positive, thus $L(c,p,-\infty) \in \left(1/2,1+\kappa_2/(2\kappa_1)\right)$. By the same argumentation, we can get limits for $L(c,p,\theta_1-) \in \left(1,1+\kappa_2/(2\kappa_1)\right)$, $L(c,p,\theta_1+) \in \left(0,\kappa_2/\kappa_1\right)$ and $L(c,p,\theta_2-) \in \left(0,1\right)$. Moreover, $L(c,p,\theta_1-)-L(c,p,\theta_1+)=1$. Let us denote restrictions of $L$ as $L_1:=L$ for $\theta \in \left(-\infty,\theta_1\right)$ and $L_2:=L$ on $\left(\theta_1,\theta_2\right)$. Then $R(L_1) \subset \left(1,1+\kappa_2/(2\kappa_1)\right)$ and $R(L_2) \subset \left(0,\kappa_2/(2\kappa_1)\right)$. Lenghts of both intervals depend on the ratio
\begin{equation}
\label{eq:ratio_kappa}
\frac{\kappa_2}{\kappa_1}=\sqrt{\dfrac{1+\sqrt{1-p}}{1-\sqrt{1-p}}}.
\end{equation}
Thus, for $\kappa_2/\kappa_1 <k$, $k \in \N$, $R(L_1) \subset \left(1,1+k/2\right)$, $R(L_2) \subset \left(0,k/2\right)$ and the number of required solutions can be at most $2 \left(\lceil k/2\rceil -1\right)$.
\end{proof}

The next assertion is a direct consequence of Lemma \ref{lemma:p>p0}. 

\begin{corollary}
\label{lemma:0_sol}
Let $c \in \left(0,\sqrt{2}\right)$, $p \in \left(16/25,{1}\right)$ and $p>p_0$. Then there is no even one-trouhged travelling wave solution of \eqref{eq:original_problem}.
\end{corollary}

Analogously, we can state the following lemma.

\begin{lemma}
\label{lemma:p<p0}
Let $c \in \left(0,\sqrt{2}\right)$, $p \in \left(0,{1}\right)$ and $p<p_0$. If $\kappa_2/\kappa_1 < k$ (cf. \eqref{eq:ratio_kappa}), $k \in \N$, then there is at most $\lceil (k+1)/2 \rceil -1$ even one-trouhged travelling wave solutions of \eqref{eq:original_problem}.
\end{lemma}

\begin{proof}
By the same \sloppy argumentation as in the proof of Lemma \ref{lemma:p>p0}, we come to the estimate \mbox{$R(L) \subset \left(0,1/2+\kappa_2/(2\kappa_1)\right)$}. For $\kappa_2/\kappa_1<k$ the length of the interval is less than $(k+1)/2$, thus there can be at most $\lceil (k+1)/2 \rceil -1$ even one-troughed solutions.
\end{proof}

\begin{remark}
If $p<p_0$ we are not able to come to the same conclusion as in Corollary \ref{lemma:0_sol} and find a region in $cp$ plane where there is no even one-trouhged travelling wave solution (except in the area with empty interval for admissible $\theta$). To sum up results of both Lemmas \ref{lemma:p>p0} and \ref{lemma:p<p0}, the inequality $\kappa_2/\kappa_1<k, \, k \in \N$, can be equivalently written as 
\begin{equation}
\label{eq:pk}
p>p_k:=1-\left(\frac{k^2-1}{k^2+1}\right)^2.
\end{equation}
Values of $p_k$ divide the $cp$ plane into strips where there are at most $2 \left(\lceil k/2\rceil -1\right)$ one-troughed travelling wave solutions for $p>p_0$ and at most $\lceil (k+1)/2 \rceil -1$ one-troughed travelling wave solutions for $p<p_0$. For the illustration see Figure \ref{fig:4}, right.
\end{remark}

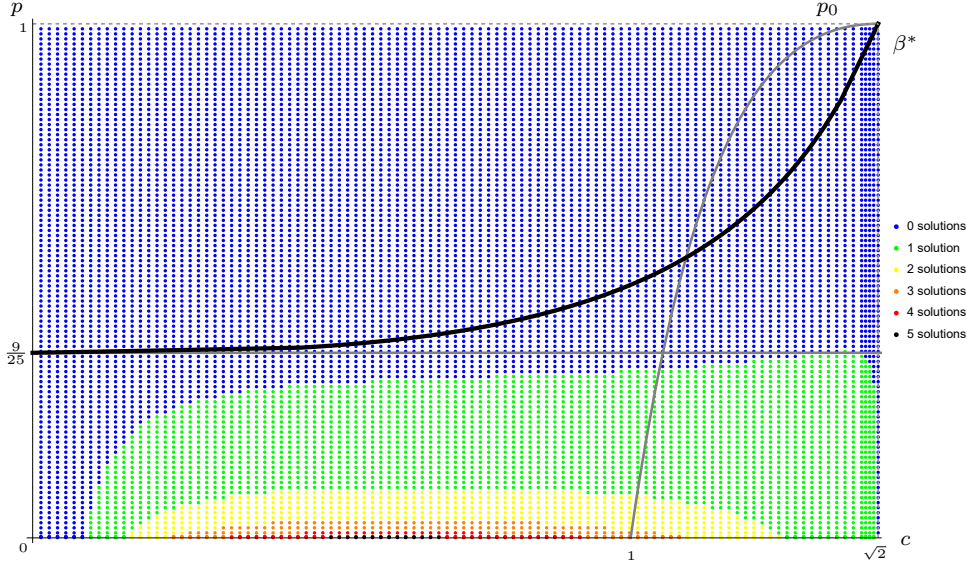
\begin{figure}[t]
\centerline{
  \setlength{\unitlength}{1mm}
  \begin{picture}(140, 70)(0,0)
    \put(10,0){\includegraphics[height=7.0cm,trim={0cm 0 0.15cm 0},clip]{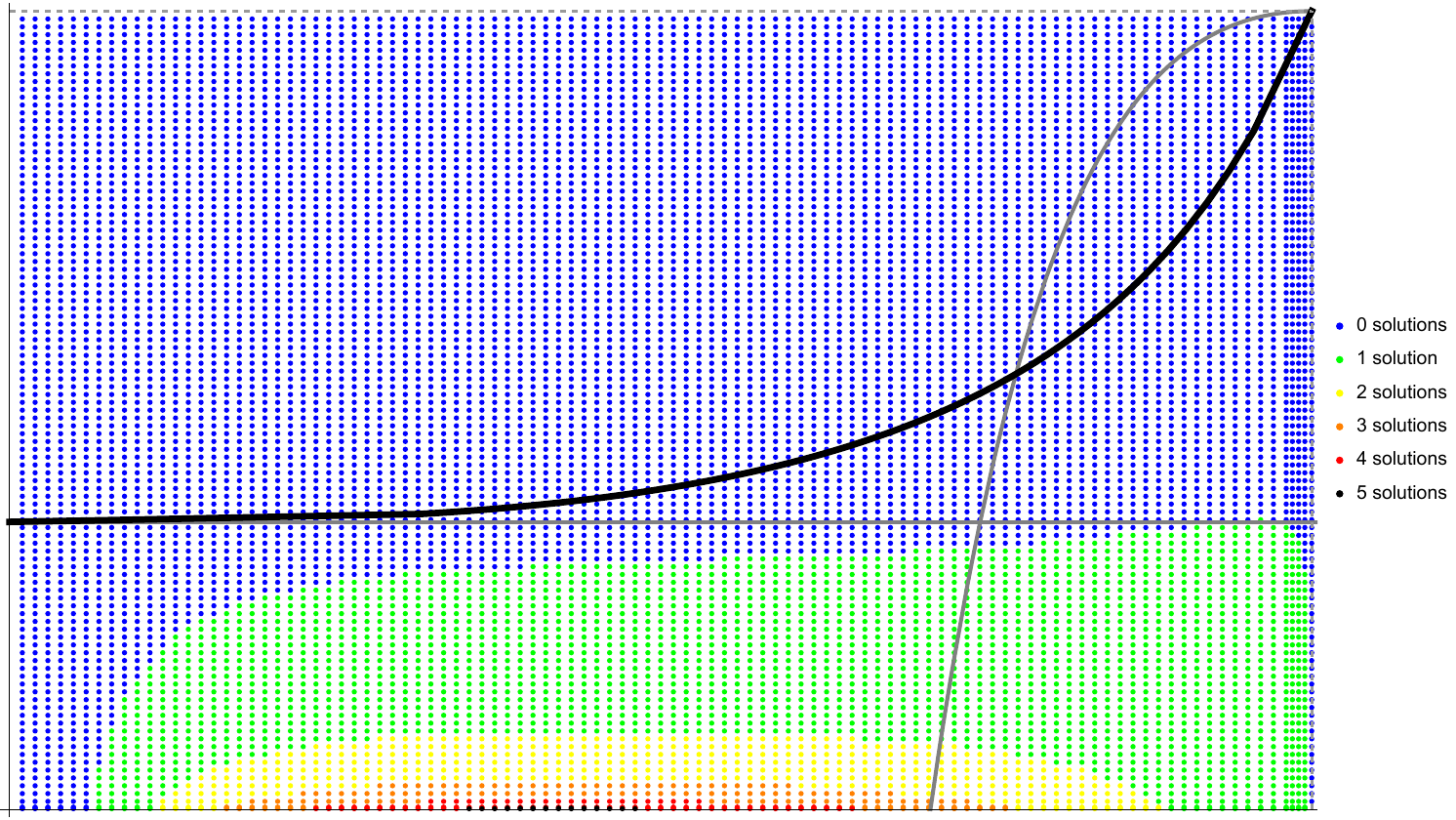}}
    \put(8,70){\makebox(0,0)[lb]{\footnotesize$p$}}         
	\put(125,0){\makebox(0,0)[lb]{\footnotesize$c$}}       
    \put(124,65){\makebox(0,0)[lb]{\footnotesize$\beta^{\ast}$}}  
    \put(114,70){\makebox(0,0)[lb]{\footnotesize$p_0$}}  
    \put(7,24){\makebox(0,0)[lb]{\tiny$\frac{9}{25}$}}   
    \put(9,-1){\makebox(0,0)[lb]{\tiny$0$}}
    \put(9,68){\makebox(0,0)[lb]{\tiny$1$}}
    \put(120,-2){\makebox(0,0)[lb]{\tiny$\sqrt{2}$}}  
    \put(89,-2){\makebox(0,0)[lb]{\tiny$1$}}   
  \end{picture}    
  }
\caption{Number of one-troughed travelling waves for particular pairs $c$ and $p$, with $\beta^{\ast}$ from \citep{holubova_leva_necesal}. For every $p<\beta^{\ast}$ there exists a travelling wave solution.}
\label{fig:2}
\end{figure}

\section{Maximum number of one-trouged solutions}
\label{sec:max}
The above stated estimates on the maximum number of solutions from Lemmas \ref{lemma:p>p0} and \ref{lemma:p<p0} say that for $p \rightarrow 0$ the number of different solutions for fixed pair $(c,p)$ could theoretically go up to infinity. However, the result for $p=0$ in \cite{champneys_mckenna} gives us the existence of at most 5 different one-troughed waves for fixed $c$. On top of that,
\begin{equation*}
\lim_{p\rightarrow 0}L(c,p,\theta)=\dfrac{1}{\pi}\arctan\left(\dfrac{c(c^2-1)\theta-c\sqrt{2-c^2}}{c^2-1+c^2\theta \sqrt{2-c^2}}\right)-\dfrac{c}{\pi}\left(\sqrt{2-c^2}+\theta_{min}(c)\right)=:\tilde{L}(c,\theta)
\end{equation*}
where transformed $\tilde{L}$ corresponds to function $L$ in Section 2 in \cite{champneys_mckenna}.
Thus, we should also expect at most 5 one-troughed solutions for $p\rightarrow 0$ in our case. This is also consistent with our visualisation, see Section \ref{sec:experiments}. In this section, we will focus on proving similar result. First, we present an auxiliary lemma.

\begin{lemma}
\label{lemma:sup_L}
Let $c \in (0,\sqrt{2})$ and $(p,\theta) \in \Omega_{p,\theta} :=(0,1)\times(\theta_{min}(c),\theta_2(c,p))$ with $(\theta_{min},\theta_2) \neq \emptyset$. Then
\begin{equation}
\label{eq:sup_L}
\sup\limits_{(p,\theta)\in \Omega_{p,\theta}} L(c,p,\theta) \leq 1 - \dfrac{c}{\pi}\left(\sqrt{2-c^2}+\theta_{min}(c)\right)=:L_{s}(c).
\end{equation}
\end{lemma}

\begin{proof}
First, from $L'_{\theta}<0$ and $L(\theta_1-)>L(\theta_1+)$ it follows that 
$\sup_{\theta}L(c,p,\theta)=\lim_{\theta \rightarrow \theta_{min}}L(c,p,\theta).$
Next, we rewrite $L(c,p,\theta_{min}^{+})$ as $L_1(c,p)+L_2(c,p)$ where
\begin{eqnarray}
\nonumber
L_1(c,p)=\dfrac{1}{2}+\dfrac{1}{\pi}\arctan\left(\dfrac{\sqrt{2}\sqrt{1+\sqrt{1-p}}(-2+2\sqrt{1-p}+pc^2(1+\theta_{min}^{+}\sqrt{2-c^2}))}{pc(2\sqrt{2-c^2}+\theta_{min}^{+}(2-c^2-c^2\sqrt{1-p}))}\right),\\
\nonumber
L_2(c,p)= \dfrac{\kappa_2}{\kappa_1}\left(\dfrac{1}{2}-\dfrac{1}{\pi}\arctan\left(\dfrac{-\sqrt{2}\sqrt{1-\sqrt{1-p}}(2+2\sqrt{1-p}-pc^2(1+\theta_{min}^{+}\sqrt{2-c^2}))}{pc(2\sqrt{2-c^2}+\theta_{min}^{+}(2-c^2+c^2\sqrt{1-p}))}\right)\right).
\end{eqnarray}
From the properties of arctangent function we get simple estimation $R(L_1) \subset (0,{1})$. It can be shown that function $L_2$ is continuous and decreasing in $p$ on $\Omega_{p,\theta}$ for any fixed $c \in (0,\sqrt{2})$. Thus 
$$\sup\limits_{p\in (0,1)}L_2(c,p,\theta_{min}^{+})=\lim_{p\rightarrow 0}L_2(c,p,\theta_{min}^{+})=-\frac{c}{\pi}(\sqrt{2-c^2}+\theta_{min}^{+}).$$
Together, we get \eqref{eq:sup_L}.
\end{proof}

Next, we introduce our main result.

\begin{theorem}
\label{th:max_number}
Let $c^2 \in (0,2)$ and $b/a \in (0,c^4/4)$. Then the problem \eqref{eq:original_problem} has at most 6 even one-troughed travelling wave solutions.
\end{theorem}

\begin{proof}
As was described above, one-troughed travelling wave solutions of \eqref{eq:original_problem} correspond to solutions of $L(c,p,\theta)=k$ with $k \in \N$ and $p = 4\xi/c^4 = 4b/(ac^4) \in (0,1)$. From Lemma \ref{lemma:sup_L} (see \eqref{eq:sup_L}) we have $\sup_{\Omega_{p,\theta}}L(c,p,\theta)\leq L_{s}(c)$. For $c \in (0,\sqrt{2})$ one can easily verify that
\begin{equation*}
\sup\limits_{c \in (0,\sqrt{2})} L_s(c)<7.
\end{equation*}
Thus, the system of equations $L(c,p,\theta)=k$, $k \in \N$, and therefore also problem \eqref{eq:original_problem} have at most 6 solutions.
\end{proof}

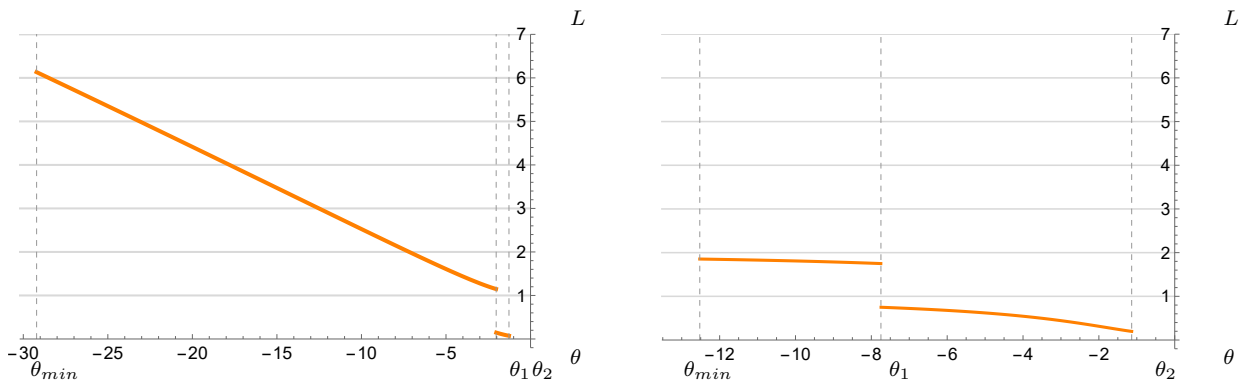
\begin{figure}[tb]
\centerline{
  \setlength{\unitlength}{1mm}
  \begin{picture}(75, 45)(0,0)
    \put(0,0){\includegraphics[height=4.5cm]{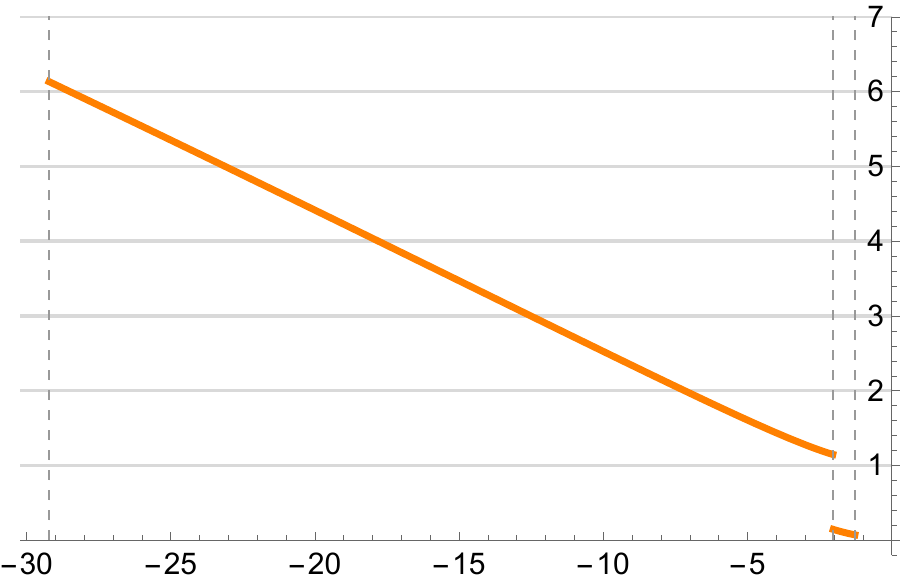}}       
	\put(74,0){\makebox(0,0)[lb]{\footnotesize$\theta$}}       
    \put(74,45){\makebox(0,0)[lb]{\footnotesize$L$}}  
    \put(3,-2){\makebox(0,0)[lb]{\footnotesize$\theta_{min}$}} 
    \put(66,-2){\makebox(0,0)[lb]{\footnotesize$\theta_1$}}  
    \put(69,-2){\makebox(0,0)[lb]{\footnotesize$\theta_2$}}         
  \end{picture}
  \hfill
       \begin{picture}(75, 45)(0,0)
    \put(0,0){\includegraphics[height=4.5cm]{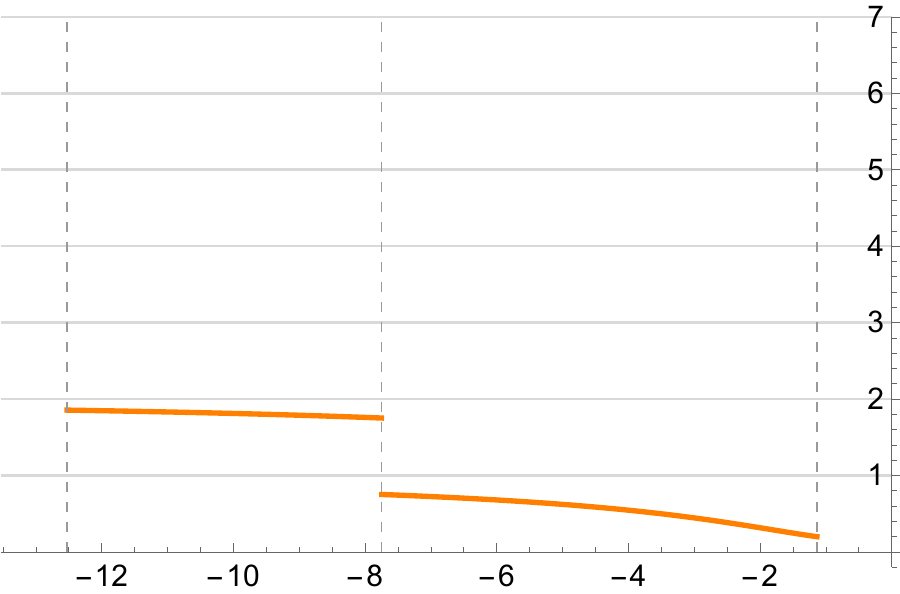}}       
	\put(74,0){\makebox(0,0)[lb]{\footnotesize$\theta$}}       
    \put(74,45){\makebox(0,0)[lb]{\footnotesize$L$}}  
    \put(3,-2){\makebox(0,0)[lb]{\footnotesize$\theta_{min}$}} 
    \put(30,-2){\makebox(0,0)[lb]{\footnotesize$\theta_1$}}  
    \put(65,-2){\makebox(0,0)[lb]{\footnotesize$\theta_2$}}         
  \end{picture}
  }
\caption{Graph of function $L$ with $c=0.61005$ and $p=0.00065$ on the left. Graph of function $L$ with $c=0.99$ and $p=0.4$ on the right.}
\label{fig:3}
\end{figure}

\section{The visualisation of the number of one-troughed waves}
\label{sec:experiments}

The visualisation of the number of one-troughed travelling waves suggests that there are between zero and five different one-troughed solutions, see Figure \ref{fig:2}. In this section, we, at least briefly, describe the algorithm for finding the exact number of one-troughed travelling wave solutions for specific fixed pair $(c,p)$ and, at the same time, create the visualisation. The procedure is as follows.

First, we cover the rectangle $\left(0,\sqrt{2}\right)\times\left(0,{1}\right)$ in $cp$ plane with a grid of points $(c,p)$ with a certain density. At each selected point we calculate the corresponding value of $\theta_{min}$ and determine whether the discontinuity $\theta_1$  lies in $(\theta_{min},\theta_2)$.

On each subinterval of $(\theta_{min},\theta_2)$ where the function $L=L(\theta)$ is continuous we determine its floor value, i.e., the lower integer part, at the end points of the relevant interval. Their difference corresponds to the number of intersections of $L$ with constant functions $f_k(\theta)=k$, $k \in \N$. In Figure \ref{fig:3}, graphs of $f_k$ are depicted in gray. We add up the number of intersections on all intervals of continuity of $L$ to obtain the number of solutions, which we then plot in different colours, see Figure \ref{fig:2}.

For the sake of completeness, we present also the graph of $L$ with specific chosen values of $c$ and $p$, in particular $c=0.61005$ and $p=0.00065$, see Figure \ref{fig:3}, left. One can see that there are intersections of $L$ with $f_k(\theta)=k$, $k=2,3,4,5,6$. Thus for this setting of parameters, there are five different symmetric one-troughed travelling wave solutions. Graphs of relevant solutions can be found in Figure~\ref{fig:figures_sol}. The interesting property of these solutions is that the larger $k$ is, the more local extrema the function $z$ has in the region under $-1$. In \cite{champneys_mckenna} this phenomenon is referred as multiwiggle type of solution.

For comparison, we also present the graph of $L$ with $c=0.99$ and $p=0.4$, see Figure \ref{fig:3}, right. In this case, there are no intersections with $f_k$ for any $k \in \N$. Thus, the problem \eqref{eq:original_problem} does not have any one-troughed travelling wave solutions with $c=0.99$ and $p=0.4$.

\section{Open questions}
\label{sec:open_questions}
Theorem \ref{th:max_number} states that  \eqref{eq:original_problem} with fixed input parameters has at most six different even one-troughed travelling wave solutions. However, the visualization suggests that, as same as in \cite{champneys_mckenna}, the maximum number of the required solutions is five (see Figure \ref{fig:2}). The proof of the nonexistence of the sixth solution is the open problem that could be solvable by a finer analysis of the range of $L$. Another related open question is the analytic description of the number of one-troughed solutions for any pair $(c,{p})$.

Moreover, we limited ourselves to the special case of symmetric (even) one-troughed solutions. The question naturally arises whether there are any asymmetric one-troughed or multi-troughed solutions. In \cite{champneys_mckenna} there is shown that there are infinitely many multi-troughed waves for some subinterval of $(0,{\sqrt{2}})$ for \eqref{eq:original_problem} with $b = 0$. There is also a conjecture stated about nonexistence of asymmetric one-troughed waves. However, it remains unproved.

\newpage
\begin{figure}[h]
\centering
\begin{subfigure}{0.4\textwidth}
	\setlength{\unitlength}{1mm}
	  \begin{picture}(65, 45)(0,0)
   		 	\put(0,0){\includegraphics[width=\textwidth]{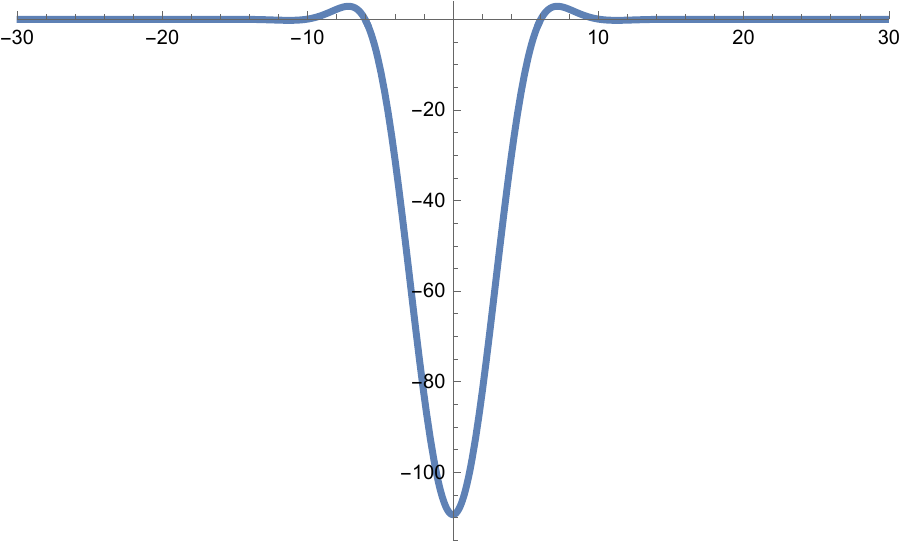}}
    		\put(31,41){\makebox(0,0)[lb]{\footnotesize$z$}} 
    		\put(65,41){\makebox(0,0)[lb]{\footnotesize$t$}}   
    	\end{picture} 
    \caption{First solution with $L=2$.}
    \label{fig:first}
\end{subfigure}
\hfill
\begin{subfigure}{0.4\textwidth}
		\setlength{\unitlength}{1mm}
	  	\begin{picture}(65, 45)(0,0)
    		\put(0,0){\includegraphics[width=\textwidth]{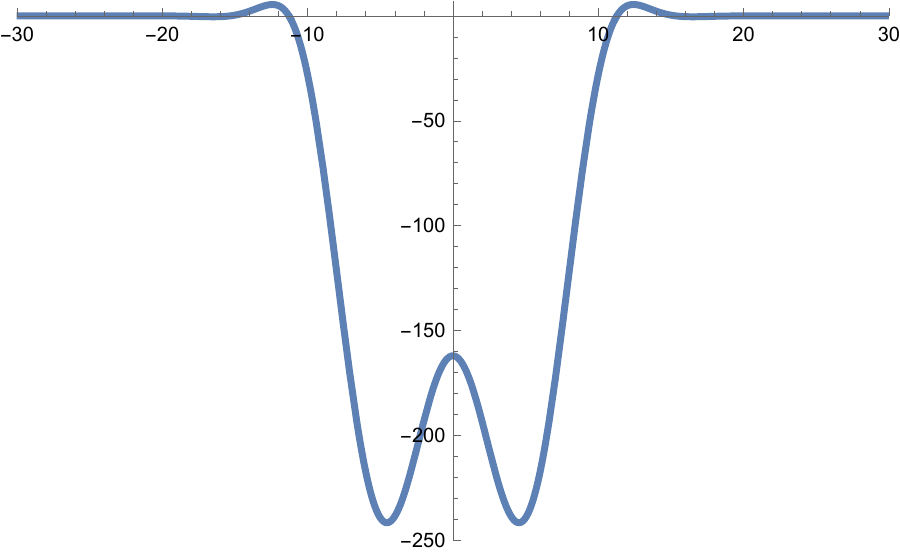}}
    		\put(31,41){\makebox(0,0)[lb]{\footnotesize$z$}} 
    		\put(65,41){\makebox(0,0)[lb]{\footnotesize$t$}} 
    		\end{picture}  
    \caption{Second solution with $L=3$.}
    \label{fig:second}
\end{subfigure}
\hfill
\vspace{0.5cm}
\begin{subfigure}{0.4\textwidth}
\setlength{\unitlength}{1mm}
	  	\begin{picture}(65, 45)(0,0)
    	\put(0,0){\includegraphics[width=\textwidth]{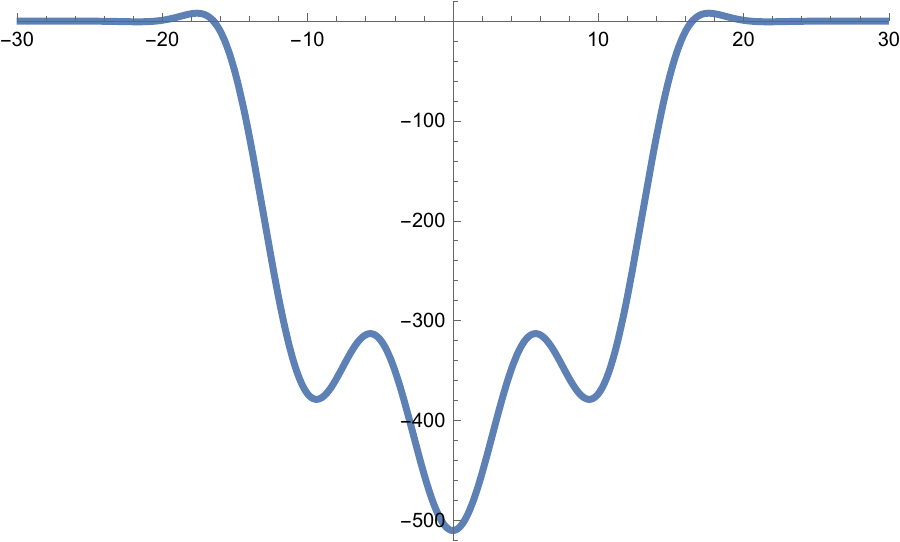}}
    	\put(31,41){\makebox(0,0)[lb]{\footnotesize$z$}} 
    		\put(65,41){\makebox(0,0)[lb]{\footnotesize$t$}} 
    		\end{picture} 
    \caption{Third solution with $L=4$.}
    \label{fig:third}
\end{subfigure}
\hfill
\begin{subfigure}{0.4\textwidth}
\setlength{\unitlength}{1mm}
	  	\begin{picture}(65, 45)(0,0)
    \put(0,0){\includegraphics[width=\textwidth]{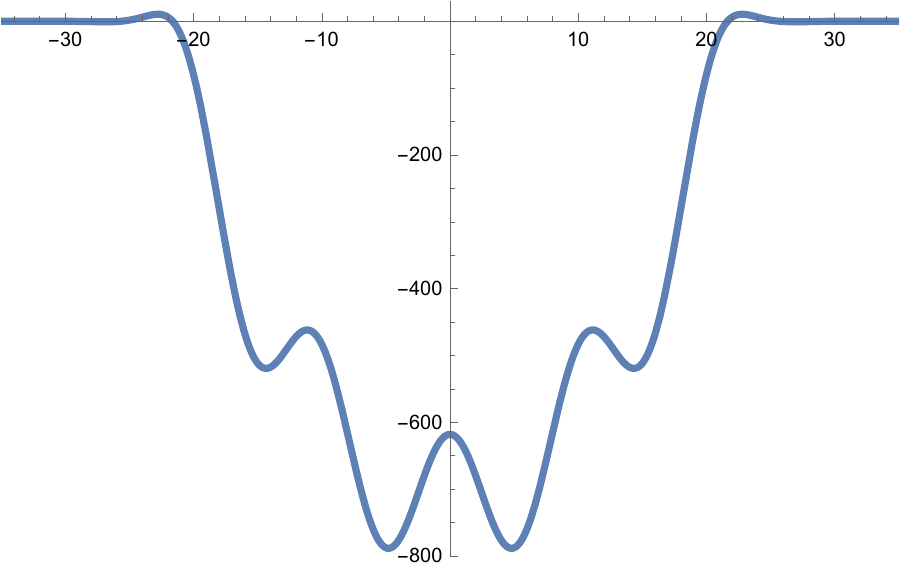}}
    \put(31,42){\makebox(0,0)[lb]{\footnotesize$z$}} 
    		\put(65,42){\makebox(0,0)[lb]{\footnotesize$t$}} 
    		\end{picture} 
    \caption{Fourth solution with $L=5$.}
    \label{fig:fourth}
\end{subfigure}
\hfill
\vspace{0.5cm}
\begin{subfigure}{0.4\textwidth}
\setlength{\unitlength}{1mm}
	  	\begin{picture}(65, 45)(0,0)
   	\put(0,0){\includegraphics[width=\textwidth]{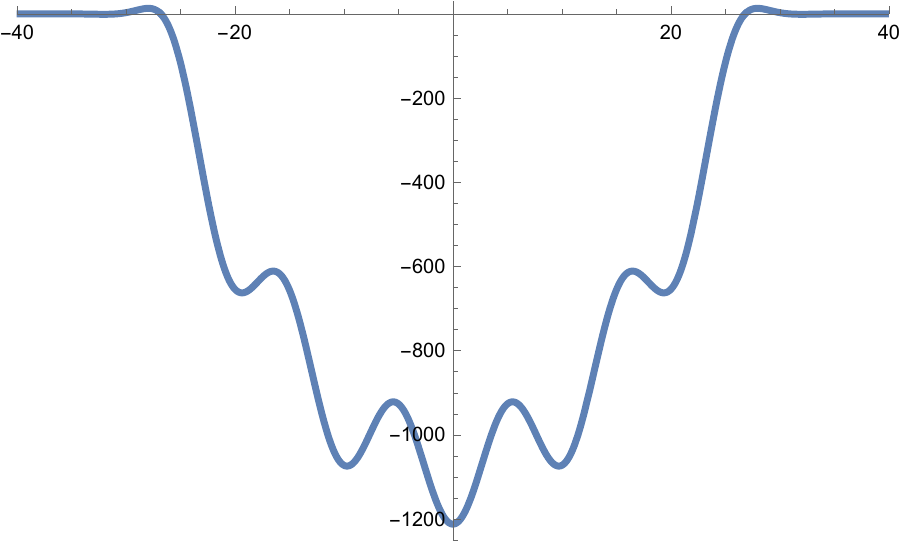}}
   	\put(31,41){\makebox(0,0)[lb]{\footnotesize$z$}} 
    		\put(65,41){\makebox(0,0)[lb]{\footnotesize$t$}} 
    		\end{picture}
    \caption{Fifth solution with $L=6$.}
    \label{fig:fifth}
\end{subfigure}        
\caption{Five different one-troughed travelling wave solutions with $c=0.61005$ and $p=0.00065$, i.e., $\xi =0.0000225069$.}
\label{fig:figures_sol}
\end{figure}
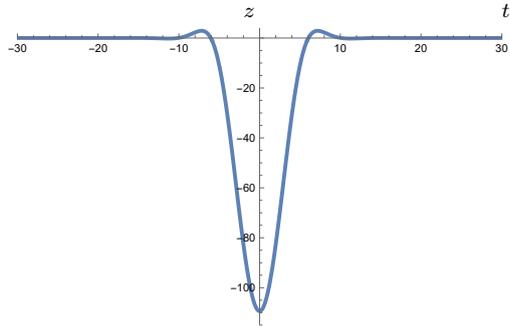
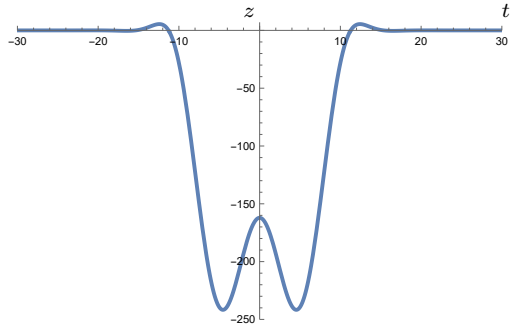
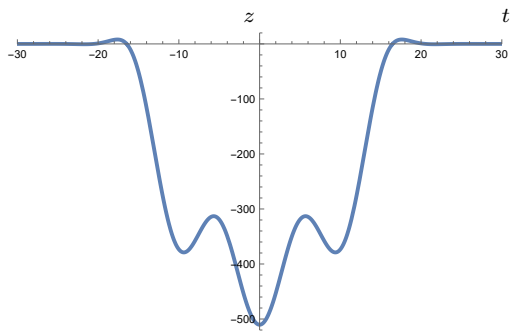
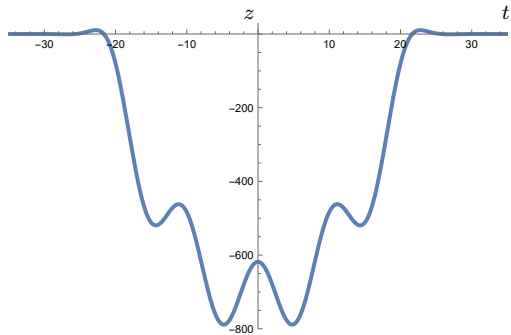
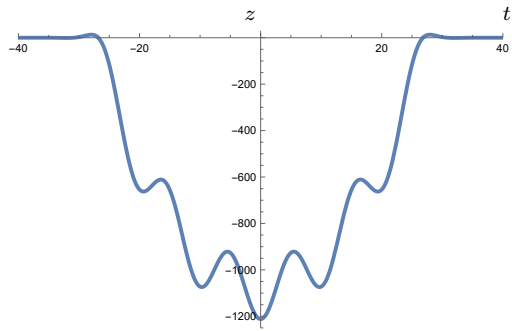

We should also mention, that there are pairs $(c,p)$ with  $p<\beta^{\ast}$, $\beta^{\ast}$ given in \cite{holubova_leva_necesal} and illustrated in Figures~\ref{fig:beta_star} and \ref{fig:2}, such that there is no one-troughed travelling wave solution. However, the existence result in \cite{holubova_leva_necesal} states that there definitely exists some solution, which means that it has to be the multi-troughed one. The question is, how many of these solutions exist for fixed $(c,{p})$.

Last but not least, the gap between the necessary and sufficient condition in \cite{holubova_leva_necesal} implies another open question: Is there any travelling wave solution in the region $p > \beta^{\ast}$? And what about their eventual (non)uniqueness?

\bigskip

\begin{acknowledgement}
	Both authors were supported by the Grant Agency of the Czech Republic, Grant No. 22--18261S.
\end{acknowledgement}
	\bibliographystyle{abbrv}
\bibliography{travelling_waves_references}
\end{document}